\documentclass[12pt]{amsart}
\usepackage{amssymb,latexsym,amsmath,amsthm,amscd}
\usepackage{setspace}
\usepackage{array}
\usepackage[all]{xy} \xyoption{arc}
\usepackage[left=3cm,top=2cm,right=3cm,bottom = 2cm]{geometry}
\usepackage{graphicx}
\usepackage[usenames,dvipsnames]{color}
\usepackage{charter}

\theoremstyle{plain}
\newtheorem{thm}{Theorem}
\newtheorem{lem}[thm]{Lemma}
\newtheorem{prop}[thm]{Proposition}

\theoremstyle{definition}

\newtheorem{ex}[thm]{Example}

\theoremstyle{remark}
\newtheorem{rmk}[thm]{Remark}

\newenvironment{pf}{\begin{proof}}{\end{proof}}

\newcommand{\veps}{\varepsilon}

\DeclareMathOperator{\uhp}{\mathcal{H}}

\DeclareMathOperator{\Tr}{Tr}

\DeclareMathOperator{\im}{im}

\DeclareMathOperator{\SL}{SL}
\DeclareMathOperator{\PSL}{PSL}

\DeclareMathOperator{\Ind}{Ind}

\DeclareMathOperator{\diag}{diag}

\newcommand*{\df}{\mathrel{\vcenter{\baselineskip0.5ex \lineskiplimit0pt \hbox{\scriptsize.}\hbox{\scriptsize.}}} =}
\newcommand*{\fd}{=\mathrel{\vcenter{\baselineskip0.5ex \lineskiplimit0pt \hbox{\scriptsize.}\hbox{\scriptsize.}}}}

\providecommand{\abs}[1]{\left\lvert#1\right\rvert}

\providecommand{\twomat}[4]{\left(\begin{matrix}#1&#2\\#3&#4\end{matrix}\right)}

\providecommand{\pseries}[2]{#1[\![ #2 ]\!]}

\newcommand{\QQ}{\mathbf{Q}}
\newcommand{\FF}{\mathbf{F}}
\newcommand{\CC}{\mathbf{C}}

\newcommand{\ZZ}{\mathbf{Z}}
\newcommand{\PP}{\mathbf{P}}

\newcommand{\be}{\textbf{e}}

\begin{document}
\title[Three-dimensional imprimitive representations of the modular group]{Three-dimensional imprimitive representations of the modular group and their associated modular forms}
\author{Cameron Franc and Geoffrey Mason}
\date{}

\begin{abstract}
This paper uses previous results of the authors \cite{FM2} to study certain noncongruence modular forms. We prove that these forms have unbounded denominators, and in certain cases we verify congruences of Atkin--Swinnerton-Dyer type \cite{ASD} satisfied by the Fourier coefficients of these forms. Our results rest on group-theoretic facts about the modular group $\Gamma$, a detailed study of imprimitive three-dimensional representations of $\Gamma$, and the theory of their associated vector-valued modular forms. For the proof of the congruences we also make essential use of a result of Katz \cite{K1}. 
\end{abstract}
\maketitle

\tableofcontents

\section{Introduction}

The main purpose of this paper is to study arithmetic properties of the Fourier coefficients of a class of noncongruence modular forms.\
In a `top-down' approach to the general problem, one takes a finite-index subgroup $G$ of $\Gamma:=\PSL_2(\ZZ)$ and considers
the spaces of weight $k$ modular forms $M_k(G, \chi)$ for characters $\chi: G\rightarrow \CC^*$ of finite order.\ The kernel $H=\ker\chi$
is typically a \emph{noncongruence} subgroup of $\Gamma$, and a form $f\in M_k(G, \chi)$ is typically a noncongruence modular form on $H$.\
 The largest subgroups $G$ (i.e., those of least index in $\Gamma$) which actually give rise to noncongruence forms in this manner are the
 nonnormal subgroups of $\Gamma$ of index $3$.\ These are conjugate to  $\bar{\Gamma}_0(2)$ (for notation, see below), and it is this case that we 
 are concerned with here.

In a well-known paper \cite{ASD}, Atkin-Swinnerton-Dyer made perhaps the first detailed study of the Fourier coefficients of noncongruence forms.\ They examined several general phenomenon, including so-called \emph{unbounded denominators} and \emph{ASD-type congruences}.\ Their observations have inspired many intersting results, cf.\ \cite{ALL1}, \cite{KL1}, \cite{KL2}, \cite{LL1}, \cite{LLY1}, \cite{L1}, \cite{Scholl}.\ In the present paper, we give a general proof of the unbounded denominators conjecture for noncongruence forms in $M_k(\bar{\Gamma}_0(2), \chi)$, and prove ASD-type congruences in certain cases.

We describe some of our results.\ The finite-order characters $\chi=\chi_{n, \veps}$ of $\bar{\Gamma}_0(2)$ are conveniently labelled by 
a positive integer $n$ and a sign $\veps=\pm 1$.\ With this notation, it transpires that $\ker\chi$ is congruence if, and only if,
$n{\mid}24$ (Theorem \ref{thmKlvl1}).\ We will show (Theorem \ref{thmgenubd}) that  if $f\in M_k(\bar{\Gamma}_0(2), \chi)$ is \emph{any} nonzero holomorphic modular form with algebraic Fourier coefficients, and if $p{\mid}n$ is a prime, then the powers of $p$ that divide the denominators of $f$ are \emph{unbounded} under any of the following circumstances: $p \geq 5$; $p=3$ and $p^2{\mid}n$; $p=2$ and $p^4{\mid}n$.

 The modular curves $X_H\df H\backslash \mathcal{H}\cup\PP^1(\QQ)$ defined by $\ker\chi$ and related groups are of  interest in themselves.\
 For example, we will show that if $n$ is odd and $\veps = -1$ then $X_H$ is a hyperelliptic curve (noncongruence unless $n{\mid}24$) given by the affine equation $y^2 = x^n+64$.\ (See Theorems \ref{thmhyperell} and  \ref{thmhypereqn} for more complete results).\ This leads to ASD-style 3-term congruences for primes $p\equiv -1$ (mod $n$) satisfied by the  coefficients of a basis of holomorphic differentials on $X_H$.
 
 The methods that we use to establish these results, advertised and illustrated (in the congruence setting) in a forthcoming
 paper \cite{FM2}, are likely to be unfamiliar to many readers, and we will say something about them here.\ Given any
form $f\in M_k(\bar{\Gamma}_0(2), \chi)$ as above, there is a $3$-dimensional \emph{vector-valued modular form} (vvmf) $F\df\ ^{t}(f_1, f_2, f_3)$ whose components $f_i$ are the forms $f|_k\gamma_i$, $\gamma_i$ ranging over cosets representatives of $\bar{\Gamma}_0(2)\backslash\Gamma$.\ Such an $F$ satisfies the transformation law
\begin{eqnarray}\label{vvmfdef}
F|_k\gamma(\tau) = \rho(\gamma)F(\tau)\ \ \ (\gamma \in \Gamma),
\end{eqnarray}
where $\rho\df \Ind_{\bar{\Gamma}_0(2)}^{\Gamma} \chi$ is the 3-dimensional representation of $\Gamma$ obtained by \emph{inducing} $\chi$, and  sits in the weight $k$ graded piece of the space
\begin{eqnarray*}
\mathcal{H}(\rho) = \oplus_{k\geq k_0} \mathcal{H}_k(\rho)
\end{eqnarray*}
of all holomorphic vector-valued modular forms on $\Gamma$ that transform as in (\ref{vvmfdef}) for some weight $k$.\ With the harmless assumption that
$\rho$ is \emph{irreducible} (equivalently, $n\not= 1$ or $3$),  the lowest nonzero weight space
is $1$-dimensional, and a good proportion of our effort is expended on understanding the nature of a spanning form
$F_0\in \mathcal{H}_{k_0}(\rho')$ where $\rho'$ is a representation \emph{equivalent} to $\rho$ with the property that $\rho'(\bar{T})$ is 
\emph{diagonal}.\  Some techniques from the theory of vector-valued modular forms (described in \cite{FM2}) show that the 3 components of the weight zero vvmf $F_0/\eta^{2k_0}$ span the solution space of a Fuchsian equation.\ This leads to the situation that the components of $F_0$
(and then also the components of \emph{every} vvmf in $\mathcal{H}(\rho)$ such as the $f$ that we started with) can be described  in terms of classical forms of level $1$ and hypergeometric series $_{3}F_2\left(a, a+1/3, a+2/3; a-b+1, a-c+1; K\right)$.\ Here, $K=1728j^{-1}$ with $j$  the absolute modular invariant, and $a, b, c$ depend on the eigenvalues of $\rho(\bar{T})$.

\medskip
The technique of relating components of vector-valued modular forms to hypergeometric series in order to study the question of
unbounded denominators was used in \cite{FM1}, where the 2-dimensional case was completely settled.\ However, in that setting one does
not encounter noncongruence modular forms.\ Chris Marks first studied the 3-dimensional case in \cite{Marks}, and obtained results about unbounded denominators for vector-valued modular forms of large enough weight.

The paper is organized as follows.\ In Section 2 we  treat the characters $\chi$ of $\bar{\Gamma}_0(2)$ and their induced representations, and in particular we determine when $ker\rho$ is congruence.\ In Section 3 we deal with the modular curves defined by $ker\chi$ and related groups, establishing that the curves we are interested in are hyperelliptic.\ In Section 4 we make a detailed study of the Fourier coefficients of the components of the lowest weight vvmf $F_0$.\ These are products of a power of $\eta$, a power of $K$, and a hypergeometric series of type $_3F_2$, and we show that each component  of $F_0$  satisfies the unbounded denominator statement of Theorem \ref{thmgenubd}.\ In Section 5, we establish the general unbounded denominator result of Theorem \ref{thmgenubd} by showing how it follows from the special case of $F_0$ together with some further general theory of vector-valued modular forms.\ In Section 6 we establish ASD-type congruences for forms in $S_{k_0}(H, \chi)$ in certain cases when $k_0=2$.\ Having available the explicit equation $y^2=x^n+64$ for the curves makes the zeta function accessible, and we can ultimately appeal to a theorem of Katz \cite{K1}.\ Interestingly, this approach \emph{eschews} our explicit formulas for the forms  in terms of hypergeometric series. It would be of interest to find a proof of our congruences using hypergeometric series in place of \cite{K1}.

\section{Imprimitive representations of dimension $3$}
We use the following notation: $\Gamma \df \PSL_2(\ZZ)$, and if $M$ is either an element or a subgroup of $\SL_2(\ZZ)$ then  $\bar M$ denotes its image in 
$\Gamma$. Thus $\bar \Gamma(n)$ and $\bar \Gamma_0(n)$ are the images of the usual congruence subgroups $\Gamma(n)$ and $\Gamma_0(n)$ in $\Gamma$.\ Given any group $G$, the notation $G'$ denotes the commutator subgroup of $G$.\ We make use of the following elements of $SL_2(\ZZ)$:
\begin{align*}
R &\df \twomat 01{-1}{-1},&  S &\df \twomat 0{-1}10,  &T&\df \twomat 1101,\\
 U &\df \twomat 10{2}{1},& V &\df TU^{-1}= \twomat {-1}1{-2}1.
\end{align*}
Before classifying the imprimitive representations of index $3$, we describe the subgroups of $\Gamma$ of index $3$.\ This is well-known, but we give it for completeness.
\begin{lem}\label{lemindex3}
  The group $\Gamma$ contains exactly four subgroups of index $3$.\ One of these is normal, and it is a congruence subgroup of level $3$. The remaining three subgroups are conjugate to $\bar \Gamma_0(2)$.
\end{lem}
\begin{pf}
Let $G \subseteq \Gamma$ be of index $3$.\ Left multiplication of $\Gamma$ on  cosets of $G$ defines a morphism of groups
$\varphi: \Gamma\rightarrow S_3$ and the image has order $3$ or $6$.\ Let $K\df \ker\varphi$.\  Then $K\subseteq G$, and either $\ker\varphi=G$, or else
$\im\varphi \cong S_3$.\ In the first case  $G\unlhd \Gamma$, so $\Gamma/G\cong \ZZ_3$ and thus $\bar T^3 \in G$.\ But $\bar \Gamma(3)$ is the normal closure of $\bar T^3$ in $\Gamma$, so $G$ contains $\bar \Gamma(3)$. 

 In the second case $\Gamma/K\cong S_3$ and $G/K$ is a subgroup of order $2$.\ Now the image of $\bar T$ in $\Gamma/K$ has order 2 or 3, and if it
is 3 then as before $K\supseteq \bar\Gamma(3)$.\ However, $\bar\Gamma/\bar\Gamma(3) \cong A_4$, and this group does not have $S_3$
as a homomorphic image.\ This shows that $\bar T$ has order 2, so $K$ contains the normal closure of $\bar T^2$,
which is $\bar\Gamma(2)$.\ Because both $K$ and $\bar\Gamma(2)$ both have index $6$ in $\Gamma$, we deduce that
$K=\bar\Gamma(2)$.\ By Sylow's theorem, $G/K$ is conjugate to $\langle \bar T \rangle K/K =\bar\Gamma_0(2)/K$, so $G$ is
conjugate to $\bar\Gamma_0(2)$.
\end{pf}

Inducing one-dimensional representations from the normal subgroup of $\Gamma$ of index $3$ gives rise to congruence modular forms, and so we will ignore these representations in what follows.\ The representations of $\Gamma$ obtained from inducing characters of a non-normal subgroup $G$ of index three are 
(up to isomorphism) independent of the choice of $G$ because (Lemma \ref{lemindex3}) all such $G$ are conjugate.\ We thus focus on those representations induced from $\bar{\Gamma}_0(2)$. 

Note that $\bar{V}\in \bar{\Gamma}_0(2)$ has order $2$.\ It is well-known that $\bar{\Gamma}(2) = \langle \bar{T}^2, \bar{U}\rangle$ is a free group of 
rank $2$.\ Then $\bar{\Gamma}_0(2) = \langle \bar{T}, \bar{U}\rangle= \langle \bar{T}^2, \bar{U}\rangle \rtimes \langle \bar{V}\rangle$, and we have the relation
\[
\bar{T}\bar{U}\bar{T}^{-1}\bar{U}^{-1} = \bar{T}^2\bar{U}^2.
\]

We assert that
$N:=\bar{\Gamma}(2)'\langle \bar{T}^2\bar{U}^2\rangle= \bar{\Gamma}_0(2)'$.\ That $N\subseteq \bar{\Gamma}_0(2)'$ follows from the displayed relation.\
Observe that $N\unlhd \bar{\Gamma}(2)$.\ Moreover $\bar{V}^{-1}\bar{T}^2\bar{V}=\bar{U}\bar{T}^2\bar{U}^{-1}\equiv \bar{T}^2$ (mod $N$),
and similarly (because $\bar{V}$ is an involution) $\bar{V}^{-1}\bar{U}^2\bar{V}=\bar{T}\bar{U}^2\bar{T}^{-1}=(\bar{T}\bar{U}\bar{T}^{-1})^2
=(\bar{T}^2\bar{U}^3)^2\equiv \bar{U}^2$ (mod $N$).\ These calculations show not only that $N\unlhd \bar{\Gamma}_0(2)$, but also that
the quotient $\bar{\Gamma}_0(2)/N$ is \emph{abelian}.\ The equality $N=\bar{\Gamma}_0(2)'$ then follows.

From the identification of $\bar{\Gamma}_0(2)$, it is immediate that $\bar{\Gamma}_0(2)/\bar{\Gamma}_0(2)' \cong \ZZ\oplus (\ZZ/2\ZZ)$,
and that we can take the two summands to be generated by (the images of) $\bar{U}$ and $\bar{V}$ respectively.

Now suppose that $\chi: \bar{\Gamma}_0(2)\rightarrow \CC^\times$ is a character of \emph{finite order}.\ Since $\chi$ factors through the abelianization of
$\bar{\Gamma}_0(2)$, it follows from the preceding discussion that there exists a primitive $n$th root of unity $\lambda$ and a sign $\veps = \pm 1$ such that 
\begin{align}
\label{chidef}
\chi(\bar{U}) &= \lambda, & \chi(\bar{V})&=\veps.
\end{align}

We now consider the induced representation 
\[
\rho \df \Ind^{\Gamma}_{\bar{\Gamma}_0(2)}\chi.
\]
Essentially by definition,  the underlying
module $V$ furnishing $\rho$ is a direct sum 
\begin{eqnarray*}
V=V_0\oplus V_1\oplus V_2
\end{eqnarray*}
 of $1$-dimensional vector spaces $V_j$, where $V_0$ affords $\chi$ and the $V_j$ are permuted transitively by $\Gamma$.\ We fix this notation for the remainder of this section.

\begin{prop}
\label{thmrho1}
With respect to an ordered basis $v_0, v_1, v_2$ of $V$ with $v_j\in V_j$, and up to a possible reordering of $V_1$ and $V_2$, we have
\begin{align*}
\rho(\bar{R})&= \left(\begin{matrix} 0 & 1 & 0 \\0 & 0 & 1 \\ 1 & 0 & 0\end{matrix}\right),&
\rho(\bar{S})&= \veps\left(\begin{matrix} 0 & 0 & \lambda \\0 & 1 & 0 \\ \bar{\lambda} & 0 & 0\end{matrix}\right),& 
 \rho(\bar{T}) &= \veps\left(\begin{matrix} \lambda & 0 & 0 \\0 & 0 & 1 \\0 & \bar{\lambda} & 0\end{matrix}\right).
\end{align*}
In particular, the eigenvalues of  $\rho(\bar{T})$ are
$\{\veps\lambda,\ \pm \sigma\}$ for some $\sigma$ such that\ $\sigma^2=\bar{\lambda}$. 
\end{prop}
\begin{pf} 
Since $\Gamma$ permutes the $V_j$ transitively and the stabilizer of $V_0$ is $\bar{\Gamma}_0(2)$, then $\bar{\Gamma}(2)$ leaves each $V_j$ invariant and the quotient $\bar{\Gamma}/\bar{\Gamma}(2)\cong S_3$ induces every possible permutation of the $V_j$.

Since $\bar{R}$ has order $3$ we may choose notation so that  $\rho(\bar{R}): v_2 \mapsto v_1\mapsto v_0\mapsto v_2$, where $v_j$ spans $V_j$.\ Then
with respect to the ordered basis $v_0, v_1, v_2$, $\rho(\bar{R})$ is as indicated.\ Now $\rho(\bar{T})=\rho(\bar{V})\rho(\bar{U})$ has $v_0$ as eigenvector with eigenvalue $\veps\lambda$.\ Moreover,  $\rho(\bar{T})$ and interchanges $V_1$ and $V_2$.\ Therefore, we have
\[
 \rho(\bar{T}) = \left(\begin{matrix} \veps\lambda & 0 & 0 \\0 & 0 & u \\0 & v & 0\end{matrix}\right)
\]
for scalars $u, v$.\ We assert that $uv=\bar{\lambda}$.\ To see this, use the relation
$\bar{R}\bar{T}^2\bar{R}^{-1}= \bar{U}^{-1}$
to obtain
\begin{align*}
\bar{T}^2V_1 &=\bar{T}^2\bar{R}^{-1}V_0 = \bar{R}^{-1}\bar{U}^{-1}V_0= \bar{\lambda}V_1.
\end{align*}
This says that the eigenvalue of $\rho(\bar{T}^2)$ on $V_1$ is $\bar{\lambda}$, and our assertion follows.\
Since $\bar{S}\bar{R}=\bar{T}$, we find that
\[
\rho(\bar{S})=\rho(\bar{T})\rho(\bar{R})^{-1} = \left(\begin{matrix} \veps\lambda & 0 & 0 \\0 & 0 & u \\0 & v & 0\end{matrix}\right)
\left(\begin{matrix}0 & 0 & 1 \\ 1 & 0 & 0 \\  0 & 1 & 0\end{matrix}\right)= \left(\begin{matrix} 0 & 0 & \veps\lambda \\0 & u & 0 \\ v & 0 & 0\end{matrix}\right),
\]
so that
\[
\rho(\bar{S}^2)=\left(\begin{matrix} \veps v\lambda & 0 & 0 \\0 & u^2 & 0 \\0 & 0 & \veps v\lambda\end{matrix}\right)=I.
\]
Since $uv=\bar{\lambda}$, we see that $u=\veps$, and the matrices representing $\bar{T}$ and $\bar{S}$ are then as in the statement of the Proposition.
\end{pf}

\begin{lem}
\label{lemmases} There is a short exact sequence
\[
1 \rightarrow A \rightarrow \rho(\bar{\Gamma}) \rightarrow S_3\rightarrow 1,
\]
where $A\df \rho(\bar{\Gamma}(2)) = \langle \rho(\bar{T^2}), \rho(\bar{U})\rangle$, and exactly one of the following holds:
\begin{enumerate}
\item[(i)] $\gcd(n, 3)=1$ and  $A \cong \ZZ_n\times \ZZ_n$;
\item[(ii)] $3{\mid}n$ and $A\cong \ZZ_n \times \ZZ_{n/3}$.
\end{enumerate}
\end{lem}
\begin{pf} From Lemma \ref{thmrho1} we obtain 
\begin{align*}
 \rho(\bar{T^2}) &= \left(\begin{matrix} \lambda^2 & 0 & 0 \\0 & \bar{\lambda} &0  \\ 0& 0& \bar{\lambda} \end{matrix}\right), &
  \rho(\bar{U}) &= \rho(\bar{R}\bar{T}^{-2}\bar{R}^{-1})=\left(\begin{matrix} \lambda & 0 & 0 \\0 & \lambda &0  \\ 0& 0& \bar{\lambda}^2 \end{matrix}\right).
\end{align*}
Now all parts of the Lemma can be checked directly.
\end{pf}

\begin{lem}\label{lemred} 
The following statements are equivalent:
\begin{enumerate}
\item[(i)] $n{\mid}3$,
\item[(ii)] $\rho$ is \emph{not} irreducible,
\item[(iii)] $\rho(\bar{T})$ has repeated eigenvalues.
\end{enumerate}
\end{lem}
\begin{pf} 
By Lemma \ref{thmrho1}, $\rho(\bar{T})$ has a repeated eigenvalue if, and only if, $\veps\lambda = \pm \sigma$.\ Because $\sigma^2 = \bar{\lambda}$
(cf.\ Lemma \ref{thmrho1}) this is equivalent to $\lambda^3=1$, so (i) and (iii) are equivalent.

Now we show that (ii) and (iii) are also equivalent.\ If $\rho$ is \emph{not} irreducible then its completely reducibility (which holds because $\rho(\Gamma)$ is finite) means that $\rho(\Gamma)$ leaves invariant some 1-dimensional subspace of $V$.\ If $\rho(\bar{T})$ has \emph{distinct} eigenvalues then such
an
invariant subspace is necessarily one of the 3 eigenspaces for $\rho(\bar T)$.\ These are spanned by
\[
\left(\begin{smallmatrix}1 \\0 \\0\end{smallmatrix}\right),\quad \left(\begin{smallmatrix}0 \\ 1 \\ \sigma\end{smallmatrix}\right),\quad \left(\begin{smallmatrix}0 \\ 1 \\ -\sigma\end{smallmatrix}\right),
\]
where $\sigma^2 = \bar\lambda$. However, from the explicit nature of the matrix $\rho(\bar S)$ given in Proposition \ref{thmrho1}, we see that \emph{none} of these three eigenspaces are invariant under $\rho(\bar S)$, and this is a contradiction.\ Thus we have established the implication (ii) $\Rightarrow$ (iii).

Conversely, if (iii) holds then we know that $\lambda^3=1$.\ Then we check directly that $\rho(\bar{S})$ and $\rho(\bar{T})$ each leave
the span of $\left(\begin{smallmatrix}1 \\ \lambda \\ \bar{\lambda}\end{smallmatrix}\right)$ invariant, in which case (ii) holds.\
This completes the proof of the Lemma.
\end{pf}

Our next theorem determines when $\rho$ factors through a congruence quotient of $\Gamma$.\ This will be important for recognizing when vvmfs are themselves congruence.
\begin{thm}\label{thmKlvl1} 
Let $\rho$ be a representation of $\Gamma$ that is induced from a one-dimensional representation $\chi$ of $\bar\Gamma_0(2)$ with finite image. Let $n$ be such that $\chi(\bar U)$ is a primitive $n$th root of unity. Then $\ker \rho$ is a congruence subgroup if, and only if, $n{\mid}24$. 
\end{thm}
\begin{pf} In the following proof, we repeatedly use the fact that if $G\subseteq \Gamma$ is a congruence subgroup, then for a positive integer $N$,
$\bar{\Gamma}(N)\subseteq G$ if, and only if, $\bar{T}^N\in G$.\ 

Set $K\df \ker\rho$.\ We assume until further notice that
 $K$ is a congruence subgroup.\ From Lemma \ref{thmrho1} it follows that $\rho(\bar{T})$ has order $2n$.\ Thus $\bar{T}^{2n}\in K$, and because $K$ is assumed to be congruence then  $\bar{\Gamma}(2n)\subseteq K$, and $K$ has level \emph{exactly} $2n$.\ If $n$ is divisible by a prime $p\geq 5$, it follows that 
$\Gamma/\bar\Gamma(2n)$ has a quotient  $\Gamma/\bar\Gamma(p)\cong \PSL_2(p)$, which is nonsolvable.\ On the other hand, this same group must also be a quotient of $\rho(\Gamma)$, which is solvable by Lemma \ref{lemmases}, contradiction.\ So $n$ is divisible only by the primes $2$ and $3$.

Suppose next that $16{\mid}n$.\ We shall derive a contradiction.\
Let $M$ be the unique normal subgroup of $\bar{\Gamma}(2)$ such that $\bar{\Gamma}(2)/M\cong\ZZ_{16}^2$.\ From the description of
$A$ in Lemma \ref{lemmases} we see that $K\subseteq M$.\ As $K$ is congruence and $\bar{T}^{32}\in M$, we have $\bar\Gamma(32)\subseteq M$.\ In effect,  this reduces us to deriving a contradiction if $n=16$ and $K=M$.\  Indeed, consider
 the tower of groups
\[
N\df\bar{\Gamma}(2) \supseteq \bar\Gamma(4) \supseteq \bar\Gamma(8) \supseteq \bar\Gamma(16)\supseteq \bar\Gamma(32)\fd N_1
\]
Note that $N/\bar\Gamma(4)\cong \ZZ_2^2$ and the other quotients satisfy $\bar\Gamma(2^f)/ \bar\Gamma(2^{f+1})\cong \ZZ_2^3$.\ The element $\bar{R}$ acts on each of these latter quotients by conjugation, with fixed-point subgroup of order $2$.\ So we have $\abs{N/N_1}=2^{11}$ and $\abs{C_{N/N_1}(\bar{R})}=2^{3}$.\ On the other hand, $\bar{R}$ acts without fixed-points on $N/M$, and $N_1\subseteq M$.\ We conclude that $K/N_1 = C_{N/N_1}(\bar{R})\unlhd N/N_1$ where $C_{N/N_1}(\bar{R})$ denotes the fixed points of $\bar{R}$ acting on $N/N_1$.\ However, one can calculate explicitly that this is \emph{false}, and this is the desired contradiction.

Assume that $9{\mid}n$.\ By Lemma \ref{lemmases} we can find a subgroup $M$ such that $K\subseteq M \subseteq \bar{\Gamma}(2)$ and $\bar{\Gamma}(2)/M\cong \ZZ_3^2$.\ Then $\bar T^6\in M$, whence also $M\supseteq \bar \Gamma(6)$.\
But then $\bar{\Gamma}(2)/M$ is a quotient of $\bar{\Gamma}(2)/\bar\Gamma(6)$ of order $9$, a contradiction because $|\bar{\Gamma}(2)/\bar\Gamma(6)|=12$.\ This completes the proof of the statement  $K$ congruence $\Rightarrow$ $n{\mid}24$.\ It remains to prove the converse. 

Writing $K_n$ to indicate the dependence of $K$ on $n$, it follows from Lemma \ref{lemmases} that $K_n\subseteq K_m$ if $m{\mid}n$.\ Therefore, it suffices to assume that $n=24$ and show that $K=K_{24}$ is congruence.\ We will actually show that
\begin{eqnarray}\label{K24identity}
K_{24}= \bar{\Gamma}(2)'\bar{\Gamma}(48).
\end{eqnarray}

By Lemma \ref{lemmases} we have $\bar{\Gamma}(2)/K_{24}\cong \ZZ_{24}\times \ZZ_8$, and in particular, there is a \emph{unique} subgroup $M$ satisfying
$K\subseteq M\subseteq \bar{\Gamma}(2)$ and $|\bar{\Gamma}(2)/M|=3$.\  The unicity of $M$ ensures that $M \unlhd \Gamma$, and one readily deduces
 that $\bar{\Gamma}(6)\subseteq M$.\ Since $\bar{\Gamma}(2)/\bar{\Gamma}(6)\cong \PSL_2(3)\cong A_4$, then 
 $M/\bar{\Gamma}(6)=(\bar{\Gamma}(2)/\bar{\Gamma}(6))'$, i.e.,
 $M=\bar{\Gamma}(6)\bar{\Gamma}(2)'$.
 
 Since $M/\bar{\Gamma}(2)'$ is a free abelian group of rank $2$
 then $K/\bar{\Gamma}(2)'$ is the \emph{unique} subgroup with quotient $\ZZ_8\times \ZZ_8$.\ Now $\bar{\Gamma}(48)\bar{\Gamma}(2)'
 \subseteq \bar{\Gamma}(6)\bar{\Gamma}(2)'\subseteq M$, so if we can also show that $M/\bar{\Gamma}(48)\bar{\Gamma}(2)'\cong \ZZ_8\times\ZZ_8$, then (\ref{K24identity}) will follow.\ From the last paragraph, this is the same as showing that 
 $\bar{\Gamma}(6)\bar{\Gamma}(2)'/\bar{\Gamma}(48)\bar{\Gamma}(2)'\cong \ZZ_8\times\ZZ_8$.
 
 We interpolate a general Lemma.
\begin{lem}
\label{lemmagen} 
Suppose that $A$, $B$, $C$ are normal subgroups of a group $G$ with $B\subseteq A$.\ There is a commuting diagram
\[\xymatrix{
1\ar[r]&\dfrac{(B\cap C)(C\cap A')}{B\cap C} \ar[r]\ar[d]_{f_1}&\dfrac{A\cap C}{B\cap C}\ar[r]\ar[d]_{f_2} &\dfrac{(A\cap C)A'}{(B\cap C)A'}\ar[r]& 1\\
1\ar[r]& \dfrac{BA'}{B}\ar[r]&\dfrac{A}{B}\ar[r] & \dfrac{A}{BA'}\ar[r] & 1
}\]
where the horizontal rows are short exact, $f_2$ is injective, and the restriction $f_1$ of $f_2$ to the indicated subgroup is also injective.\ In
particular, if $f_2$ is an \emph{isomorphism} it induces an isomorphism of short exact sequences, and especially an isomorphism
\begin{equation}
\label{iso}
\frac{(A\cap C)A'}{(B\cap C)A'} \stackrel{\cong}{\rightarrow} \frac{A}{BA'}.
\end{equation}
\end{lem}
\begin{pf} The lower short exact sequence is canonical.\ As for the upper sequence, 
the map $a(B\cap C) \mapsto a(B\cap C)A'\ (a\in A\cap C)$ induces a surjective morphism of groups $(A\cap C)/(B\cap C)\rightarrow (A\cap C)A'/(B\cap C)A'$ with kernel $(A\cap C)\cap (B\cap C)A'/(B\cap C)=(B\cap C)(A\cap C\cap A')/(B\cap C)=(B\cap C)(C\cap A')/(B\cap C)$, so that
the upper sequence is also short exact.

Next, the map $a\mapsto aB\ (a\in A\cap C)$ induces a group morphism $A\cap C\rightarrow A/B$ with kernel
$A\cap C\cap B=B\cap C$.\ Thus the middle vertical arrow $f_2$, defined by $f_2(a(B\cap C))=aB$, is injective, and
 we easily check that $f_1$, which is the restriction of $f_2$
 to $(B\cap C)(C\cap A')/(B\cap C)$, is also injective.\ 
 
 It is clear that the diagram commutes, so suppose now that $f_2$ is an isomorphism.\ Since $BA'/B = (A/B)'$ and $(A\cap C)A'/(B\cap C)A'$
 is \emph{abelian} then $f_1$ is necessarily \emph{surjective}.\ Since it is also injective it is therefore an isomorphism, and the final assertions
 of the Lemma follow immediately.
 \end{pf}
 
 We apply the Lemma with $A$, $B$, $C$ equal to $\bar{\Gamma}(2)$, $\bar{\Gamma}(16)$, $\bar{\Gamma}(3)$ respectively.\ Because
 $A\cap C=\bar{\Gamma}(6)$ and $B\cap C=\bar{\Gamma}(48)$, (\ref{iso}) then reads
 \begin{eqnarray*}
\frac{\bar{\Gamma}(6)\bar{\Gamma}(2)'}{\bar{\Gamma}(48)\bar{\Gamma}(2)'} \cong \frac{\bar{\Gamma}(2)}{\bar{\Gamma}(16)\bar{\Gamma}(2)'}.
\end{eqnarray*}
In order to complete the proof of the Theorem, it is therefore sufficient show that $\bar{\Gamma}(2)/\bar{\Gamma}(16)\bar{\Gamma}(2)'\cong \ZZ_8\times\ZZ_8$. (In effect, we have reduced the proof to the case $n=8$.)

Let $G\df \bar{\Gamma}(2)/\bar{\Gamma}(16) =\langle \bar{T}^2\bar{\Gamma}(16),\bar{U}\bar{\Gamma}(16) \rangle$.\ 
We have $\abs{G}=2^8$.\ We calculate that 
\[
C\df[T^2, U] = \twomat {21}{-8}8{-3},
\]
that the image of $\bar{C}$ in $G$ has order $4$, and that $[C, T^2] \equiv [C, U] \equiv I\pmod{16}$. Thus $\langle \bar{C}\bar{\Gamma}(16)\rangle = G'$ has order $4$, so that $G/G' \cong \bar{\Gamma}(2)/\bar{\Gamma}(2)'\bar{\Gamma}(16)$
is abelian with 2 generators, exponent $8$ and order $2^6$.\ Therefore
 $\bar{\Gamma}(2)/ \bar{\Gamma}(2)'\bar{\Gamma}(16)\cong \ZZ_8\times\ZZ_8$, and the proof of Theorem \ref{thmKlvl1} is complete.
 \end{pf}
 
\section{Geometric considerations}
\label{s:geometry}
We retain previous notation, in particular, $\chi{:} \bar{\Gamma}_0(2)\rightarrow \CC^\times$ is the linear character described in (\ref{chidef}), and we put $\rho= \Ind_{\bar{\Gamma}_0(2)}^{\Gamma}\chi$, $K=\ker\rho$ and $A= \bar{\Gamma}(2)/K$. We also set
\begin{align*}
H&\df \ker\chi, & H_0&\df H\cap \bar{\Gamma}(2), & H_1&\df H_0\langle \bar{V}\rangle,& H_2&\df H\langle \bar{V}\rangle.
\end{align*}
We collectively refer to these groups as the \emph{$H$-groups}. Note that the group $H$, and thus all of the $H$-groups, does not depend on the particular primitive $n$th root of unity $\lambda = \chi(\bar U)$, it only depends on its order $n$ and the sign $\veps$ satisfying $\chi(\bar V) = \veps$. The next fact is only slightly less obvious.
\begin{lem}
\label{lem:H0}
  The group $H_0$ depends only on the order $n$ of $\chi(\bar U)$.
\end{lem}
\begin{proof}
Note that $\ker \chi$ is generated by $\bar U^n$, $\Gamma_0(2)'$, and $\bar V^m$ where $m = 0$ or $1$. We thus wish to show that 
\[\langle \bar U^n,\bar V\rangle \bar \Gamma_0(2)'\cap \bar \Gamma(2) = \langle \bar U^n\rangle \bar \Gamma_0(2)'\cap \bar \Gamma(2).\]
Now $\bar{U}\in\bar{\Gamma}(2)$, so the group on the right is just $\langle \bar U^n\rangle \bar \Gamma_0(2)'$.\ Similarly,
the group on the left can be written as $\langle \bar{U}^n\rangle \bar{\Gamma}_0(2)'\langle \bar{V}\rangle \cap \bar{\Gamma}(2)=
\langle \bar{U}^n\rangle \bar{\Gamma}_0(2)'\langle (\bar{V}\rangle \cap \bar{\Gamma}(2))=\langle \bar{U}^n\rangle \bar{\Gamma}_0(2)' $.\
This completes the proof of the Lemma.
\end{proof}

We are interested in the algebraic curves defined by the $H$-groups.\ In order to study these objects, we require more group-theoretic data about the $H$-groups themselves. Notice that $\bar{\Gamma}_0(2)'\subseteq H_0$.\ Thus the $H$-groups are each \emph{normal} in $\bar{\Gamma}_0(2)$ with abelian quotient.\
Their precise relation to each other depends on $\veps$ and the parity of $n$, and we explain this first.
\begin{lem}
\label{lemABCD} 
One of the following holds:
\begin{enumerate}
\item[A.] $n$ is \emph{odd}, $\veps=-1$, $H=H_0$ and $H_1=H_2$;
\item[B.] $n$ is \emph{odd}, $\veps=+1$ and $H=H_1=H_2$;
\item[C.] $n$ is \emph{even}, $\veps=-1$, all four of the $H$-groups are distinct and $\abs{H_2/H_0}=4$;
\item[D.] $n$ is \emph{even}, $\veps=+1$ and $H=H_1=H_2$.
\end{enumerate}
Moreover, in all cases $\abs{H_1/H_0} = 2$. 
\end{lem}
\begin{pf} 
Because $\bar{V}\in \bar{\Gamma}_0(2)\setminus{\bar{\Gamma}(2)}$ we always have $\abs{H_1/H_0}=2$.\
If $\veps=+1$ then $\chi(\bar{V})=1$, so $\bar{V}\in H\setminus{H_0}$ and B or D holds according to the parity of $n$.

If $n$ is \emph{odd} and $\veps=-1$ then from (\ref{chidef}) we see immediately that $\ker\chi\subseteq \bar{\Gamma}(2)$, i.e.,
$H=H_0$.\ Then A holds.\ 

Finally, if $n$ is \emph{even} and $\veps=-1$ then we again use (\ref{chidef}) to see that $\ker\chi$ contains an element in 
$\bar{\Gamma}_0(2)\setminus{\bar{\Gamma}(2)}$, whereas $\bar{V}\notin \ker\chi$.\ Pictorially, there is a diagram of containments of index $2$:
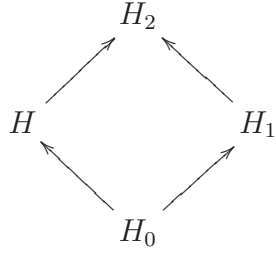
\[\xymatrix{
&H_2&\\
H\ar[ru] && H_1 \ar[lu]\\
&H_0\ar[lu]\ar[ru]&
}\]
This is case C.
\end{pf}

In what follows, we refer to the different possibilities of the preceding Lemma as Case A, Case B, etc.

\begin{lem}\label{lemH0index} The following hold.
\begin{enumerate}
\item[(i)]$\abs{\bar{\Gamma}(2)/H_0}=\abs{\bar{\Gamma}_0(2)/H_1}=n$.
\item[(ii)] $\bar{\Gamma}(2) = H_0\langle \bar{U}\rangle = H_0\langle\bar{T}\bar{U}\bar{T}^{-1}\rangle$.
\item[(iii)] $\abs{\bar{\Gamma}(2)/H_0\langle \bar{T}^2\rangle} \leq 2$ with equality if and only if $n$ is \emph{even}.
\end{enumerate}
\end{lem}
\begin{pf} Since the abelianization of $\bar{\Gamma}_0(2)$ is generated by the images of $\bar{V}, \bar{U}$ and
$\bar{V}\notin\bar{\Gamma}(2)$, it
follows that $\bar{\Gamma}(2)=H_0\langle \bar{U}\rangle$.\ Then because $\chi(\bar{U})=\lambda$ is a primitive $n$th root of unity, the
first equality of part (i) of the Lemma follows immediately.\ The second follows because $\abs{H_1/H_0}=2$ (Lemma \ref{lemABCD}).

We have already established the first equality in part (ii).\ Then because $\bar{V}$ normalizes both $\bar{\Gamma}(2)$ and $H_0$, and since $\bar{V}\bar{U}\bar{V}^{-1}=\bar{T}\bar{U}\bar{T}^{-1}$, the second equality  also holds.

Finally, the equality $\chi(\bar{T}^2)=\lambda^2$ 
 follows from Lemma \ref{thmrho1}.\ Part (iii) is an immediate consequence of this together with part (i).\ This completes the proof of the Lemma.
\end{pf}

\begin{lem}\label{leminvolutions} 
The group $\bar{\Gamma}_0(2)$ contains a \emph{unique} conjugacy class of involutions, and $H_1$ contains $n$ conjugacy classes of involutions.\ In Case C, $H_2$ contains $n/2$ conjugacy classes of involutions.
\end{lem}
\begin{proof}
It is well-known that $\bar\Gamma_0(2)$ has a \emph{unique} conjugacy class of involutions, but here's a proof: it is even more well-known that $\Gamma$ has a unique such conjugacy class.\ So if  $x, y\in\bar\Gamma_0(2)$  are two involutions, there is $g\in\Gamma$ such that $g^{-1}xg=y$.\ Since $\bar\Gamma_0(2)=\bar\Gamma(2)\langle x\rangle = \bar\Gamma(2)\langle y \rangle$, it follows that $g$ normalizes $\bar\Gamma_0(2)$.\ But this latter group is self-normalizing
(cf.\ Lemma \ref{lemindex3}), whence $g\in\bar\Gamma_0(2)$.\ Thus $x$ and $y$ are conjugate in $\bar{\Gamma}_0(2)$ and the proof is complete.

Now  $\bar{V}\in H_1 \unlhd \bar{\Gamma}_0(2)$.\ Since $\bar{V}$ is an involution, \emph{all} involutions of $\bar\Gamma_0(2)$ are contained in 
$H_1$ by the first paragraph.\ Involutions of $\bar\Gamma$ being self-centralizing, it follows that the conjugation action of $H_1$ on its involutions falls into 
$|\bar{\Gamma}_0(2)/H_1|$ classes that are themselves transitively permuted by $\bar\Gamma_0(2)$.\ Since 
$\abs{\bar{\Gamma}_0(2)/H_1}=n$ by Lemma \ref{lemH0index}
 then $H_1$ has $n$ classes of involutions.
 
  In Case C, the identical proof applies with $H_2$ in place of $H_1$.\ The only difference is that $H_2$ has index $n/2$ in $\bar{\Gamma}_0(2)$, so that it has
$n/2$ classes of involutions.\ This completes the proof of the Lemma.
\end{proof}

\begin{lem}\label{lemABDellpts} In cases A, B and D we have $\bar{\Gamma}_0(2)/H_1=\{\bar{U}^jH_1,\ 0\leq j\leq n-1\}$.\
Representatives for the elliptic points of $H_1$ are the numbers
\begin{eqnarray*}
\{U^{-j}(1+i)/2 \mid 0\leq j \leq n-1 \} = \left\{\frac{1}{-2j+(1-i)}\mid 0\leq j \leq n-1\right\}.
\end{eqnarray*}
\end{lem}
\begin{pf} By Lemma \ref{lemH0index} we have $|\bar{\Gamma}_0(2)/H_1|=n$.\ Now each $\bar{U}^j\in \bar{\Gamma}_0(2)$, and since $\chi(\bar{U}^j)=\lambda^j\not=1$ for $0\leq j \leq n-1$ then these elements
give \emph{distinct} representatives of $\bar{\Gamma}_0(2)/H$.\ So if $H=H_1$ we are done.\ This handles cases B and D by Lemma \ref{lemABCD}.\ In case A,
 we have $H_1=H\langle \bar{V}\rangle$, and if some $\bar{U}^j \in H\bar{V}$ then $\lambda^j=\chi(\bar{U}^j)=\chi(\bar{V})=-1$, and this is
 impossible because $n$ is \emph{odd} in case A.\ This completes the proof of the first assertion of the Lemma in all three cases.
 
 It follows from this and the discussion presented in the course of the proof of Lemma \ref{leminvolutions}, that the $n$ conjugacy classes of involutions in $H_1$ 
 have representatives $\bar{U}^{-j}\bar{V}\bar{U}^{j}\ (0\leq j\leq n-1)$.\ Since $\bar{V}$ itself has fixed-point $\frac{1+i}{2}$, it follows that the
 elliptic points of $H_1$ are represented by $\bar{U}^{-j}\left(\frac{1+i}{2}\right)$, and the Lemma is proved.
\end{pf}
Now we are ready to study the curves.\ If $B \subseteq \Gamma$ is a subgroup of finite index, we write $X_B$ for the projective algebraic curve whose complex points are identified with $B\backslash \uhp \cup \PP^1(\QQ)$. In the special cases $B = \bar \Gamma(N)$ or $\bar{\Gamma}_0(2)$, we denote the corresponding curves by $X(N)$ and $X_0(N)$, respectively.\ The \emph{genus} of $X_B$ is denoted by $g_B$.\ 
 
 Containments among the $H$-groups define various natural coverings of degree $2$ among the corresponding curves (cf.\ Lemma \ref{lemABCD}):
 \begin{align*}
X_{H_0}&\rightarrow X_{H_1}  & \mbox{(Cases A, B, D)}
\end{align*}
and
\begin{equation}
\label{diagcover}
\begin{aligned}
\xymatrix{
&X_{H_0}\ar[rd]\ar[ld]&&\\
X_H\ar[rd] && X_{H_1} \ar[ld] & \mbox{(Case C)}\\
&X_{H_2}&&
} 
\end{aligned}
\end{equation}
We can now state the first main result of this Section.
 \begin{thm}\label{thmhyperell} 
The curve $X_{H_1}$ has two cusps $\{\infty, 0\}$ and genus 0.\ In particular, $X_{H_0}$ is a \emph{hyperelliptic curve}.
\end{thm}
 
To prove this we must translate the group-theoretic facts about the $H$-groups established earlier into geometric facts about the corresponding algebraic curves.

\begin{lem}\label{lemcuspodd} 
If $n$ is \emph{odd}, $X_{H_0}$ has  three cusps $\{\infty, 0, 1\}$ and genus $(n-1)/2$.
\end{lem}
\begin{proof}
Recall that $\bar\Gamma(2)$ has three cusps, with representatives given by $\infty$, $0$ and $1$.\
These cusps are stabilized by $\bar{T}, \bar{U}$ and $\bar{T}\bar{U}\bar{T}^{-1}$ respectively, and
because we are assuming that $n$ is odd, Lemma \ref{lemH0index}(ii), (iii)
makes it clear that $H_0$ has the \emph{same} three cusps as $\bar\Gamma(2)$.

Because $\bar\Gamma(2)$ is free, the covering $X_{H_0} \to X(2)$ is regular of degree $n$ away from the cusps, and fully ramified at the cusps.\ By the Riemann-Hurwitz theorem, we see that
the genus  of $H_0$ is $(n-1)/2$, as claimed. 
\end{proof}

Similarly, we have
\begin{lem}\label{lemcuspeven}
If $n$ is \emph{even}, $X_{H_0}$ has four cusps $\{0, 1, 1/2, \infty\}$ and genus $(n-2)/2$.
\end{lem}
\begin{pf} The points $0$ and $1$ are distinct  cusps of $H_0$ by the same proof as before (using Lemma \ref{lemH0index}(ii)).\ Part (iii) of the same Lemma
shows that when $n$ is \emph{even},  $\infty$ behaves slightly differently, because now
$\bar\Gamma(2) = H_0\langle \bar{T}^2\rangle \cup H_0\langle \bar{T}^2\rangle \bar{U}$, so that
\begin{eqnarray*}
\bar\Gamma(2)\cdot\infty = (H_0\langle \bar{T}^2\rangle \cup H_0\langle \bar{T}^2\rangle \bar{U})\cdot\infty =H_0\cdot\infty \cup H_0\cdot\frac{1}{2}.
\end{eqnarray*}
This shows that $H_0$ has four cusps $\{0, 1, 1/2, \infty\}$.\ Moreover, in the covering $X_{H_0}\rightarrow X(2)$, the cusps $0$ and $1$ are fully ramified
and the ramification index at the other two cusps is $n/2$.\ The  Riemann-Hurwitz theorem again implies the desired genus formula.
\end{pf}

We turn to the proof of Theorem \ref{thmhyperell}.\ The relevant picture is as follows:
\begin{equation}
\label{covdiag1}
\begin{gathered}
\xymatrix{X_{H_0} \ar[rd]^n \ar[dd]_2 &\\
& X(2) \ar[dd]^2\\
X_{H_1} \ar[rd]_n &\\
& X_0(2)}
\end{gathered}
\end{equation}
where the arrow labels denote degrees of maps (cf.\ Lemma \ref{lemH0index}).\ We will apply the Riemann-Hurwitz formula to the
left vertical covering.\ The details are slightly different according to the parity of $n$, so let us first assume that $n$ is \emph{odd}.\
Then by Lemma \ref{lemcuspodd}, $H_0$ has  cusps $\{\infty, 0, 1\}$ and genus $(n-1)/2$.\
 As $H_1=H_0\langle \bar{V} \rangle$ and $\bar{V}$ exchanges the cusps $\{0, 1\}$, $X_{H_1}$ has two cusps
 and  only the infinite cusp of $X_{H_0}$ ramifies.\
 The $n$ elliptic points of $H_1$ (enumerated in Lemma \ref{lemABDellpts}) as well as $\infty$, each make a contribution of $1$ to $\sum_P(e_P-1)$, where $e_P$ is the ramification degree at a point $P\in X_{H_0}$.\ Therefore, 
\begin{eqnarray*}
&&2-2g_{H_0} = 2(2-g_{H_1})- \sum_P (e_P-1)\ \Rightarrow  3-n=4-2g_{H_1}-(n+1),
\end{eqnarray*}
and the desired result $g_{H_1}=0$ follows.

Now suppose that $n$ is \emph{even}.\
By Lemma \ref{lemcuspeven}, $H_0$ has  cusps $\{\infty, 0, 1, 1/2\}$ and genus $(n-2)/2$.\
In this case, we can check directly that $\bar{V}$ exchanges the cusps $\{0, 1\}$ and $\{\infty, 1/2\}$.\ 
So the only ramification arises from the $n$ elliptic points of $X_{H_1}$ (Lemma \ref{leminvolutions} still applies), and  we now obtain
\begin{eqnarray*}
 4-n=4-2g_{H_1}-n,
\end{eqnarray*}
 leading once again to $g_{H_1}=0$.\ This completes the proof of the Theorem. $\hfill \Box$

A similar argument shows 
\begin{thm} 
\label{t:casec}
In Case C, $X_H$ is a hyperelliptic curve of genus $[n/4]$.
\end{thm} 
\begin{pf} (Sketch).\ Refer to (\ref{diagcover}) for the picture in this case.\  Since $X_{H_1}$ has genus $0$ by Theorem \ref{thmhyperell}, the same is true of 
$X_{H_2}$.\ Therefore,
$X_H$ is hyperelliptic.\ $X_{H_0}$ has $4$ cusps $\{\infty, 0, 1, 1/2\}$, $H=H_0\langle \bar{U}^{n/2}\bar{V}\rangle$, and $\bar{U}^{n/2}\bar{V}$ exchanges
the cusps $\{0, 1\}$.\ $H$ contains \emph{no} involutions, hence it is free, so the only ramification in the covering $X_{H_0}\rightarrow X_H$ occurs at 
the cusps, and there are $f\df0$ or $2$ ramified cusps.\ Now  Riemann-Hurwitz yields
\[
4-n = 2(2-2g_H) - f\Rightarrow g_H= (n-f)/4 = [n/4].
\]
\end{pf}

We end this section by describing algebraic equations for the hyperelliptic curves discussed in Theorem \ref{thmhyperell}. Let $\rho = \Ind_{\bar \Gamma_0(2)}^\Gamma \chi$ be a representation as above, and assume that we are in Case A, B or D of Lemma \ref{lemABCD}. The diagram (\ref{covdiag1}) applies, and the discussion of the previous section tells us that the curve $X_0(2)$ has two cusps $\{0,\infty\}$ and similarly for $X_{H_1}$. The bottom map in (\ref{covdiag1}) is ramified only at the cusps, and it is of degree $n$. Both are curves of genus $0$. A hauptmodul for $\Gamma_0(2)$ is given by
\[
  K_0(\tau) = \frac{\eta(\tau)^{24}}{\eta(2\tau)^{24}},
\]
and one can use the transformation properties of the Dedekind eta function $\eta$, and the factorization formula for $\eta$, to verify that $K_0(0) = 0$ and $K_0(\infty) = \infty$. It follows that $\sqrt[n]{K_0}$ is a hauptmodul for $H_1$.\footnote{Compare this result with Theorem \ref{thmKlvl1}.}

In order to find an algebraic equation for the hyperelliptic curve $X_{H_0}$ we must understand the ramification of the projection map $X_{H_0} \to X_{H_1}$ of degree two. If $n$ is odd then this map is ramified at the cusp $\infty$ and at the elliptic points of $H_1$, while if $n$ is even then this map is only ramified at the elliptic points. We have seen that the set of elliptic points is equal to the following set:
\[
\left\{U^{-j}\left(\frac{1+i}{2}\right) \mid 0\leq j \leq n-1\right\}.
\]
Since $K_0$ is invariant under $\Gamma$, it takes the same value on all of these elliptic points. Further, since the powers of $\bar U$ give a set of coset representatives for $H_1\backslash\bar \Gamma_0(2)$, it follows that the values $\sqrt[n]{K_0}(U^{-j}(1+i)/2))$ are the $n$ distinct $n$th roots of $K_0((1+i)/2)$ and $y^2 = x^n - K_0((1+i)/2)$ is an affine algebraic equation for the hyperelliptic curve $X_{H_0} \to X_{H_1}$. By the theory of complex multiplication, one knows that $K_0((1+i)/2)$ is a rational number. In fact, one can use the Chowla-Selberg formula to show that $K_0((1+i)/2) = -64$. This shows that in cases A, B and D, the curve $X_{H_0}$ can be described by the affine hyperelliptic equation $y^2 = x^n+64$.\ We summarize some of our conclusions in
\begin{thm}\label{thmhypereqn} Assume that one of Cases A, B or D of Lemma \ref{lemABCD} holds.\ Then $X_{H_1}$ has genus $0$, and the double cover
$X_{H_0}\rightarrow X_{H_1}$ has genus $\left[\frac{n-1}{2}\right]$ and is described by the affine hyperelliptic equation $y^2 = x^n+64$. $\hfill \Box$
\end{thm}

\begin{rmk}
Ultimately we will only prove ASD-style congruences for modular forms arising in Case A, so we have not determined what hyperelliptic curves arise in Theorem \ref{t:casec}. 
\end{rmk}

\section{Vector-valued modular forms of minimal weight}
\label{s:newvvmfs}
Section 4.3 of \cite{FM2} explains how generalized hypergeometric series and the free-module theorem of \cite{MM} allow one to describe the module of vector-valued modular forms associated to an irreducible three-dimensional representation of $\Gamma$.\ The answer is expressed in terms of the exponents of the eigenvalues of $\rho(\bar T)$. 

Let $\rho=\Ind_{\bar{\Gamma}_0(2)}^{\bar{\Gamma}}\chi$ be as in Theorem \ref{thmrho1}, and write $\chi(\bar U) = e^{2\pi i r/n} = \lambda$ where $0< r < n$ and $\gcd(r,n) = 1$. Let $\rho'$ be a representation of $\Gamma$ equivalent with $\rho$ such that
\begin{eqnarray*}
\rho'(\bar T) = \diag(\veps \lambda,\sigma,-\sigma),
\end{eqnarray*}
where $\sigma^2 = \bar \lambda$. Define $\be(z) = e^{2\pi i z}$. Then we have the following cases: 
\[
\rho'(\bar T) = \begin{cases}
\be(\diag(\frac{2r+n}{2n},\frac{n-r}{2n},\frac{2n-r}{2n})) & \textrm{Cases A and C},\\
\be(\diag(\frac{2r}{2n},\frac{n-r}{2n},\frac{2n-r}{2n})) & \textrm{Cases B and D}.
\end{cases}
\]
Not all of the exponents above necessarily lie between $0$ and $1$.\ Since we've chosen $r$ to satisfy $0 < r < n$, it follows that $r/n$, $(n-r)/(2n)$ and$(2n-r)/(2n)$ all lie in $[0,1)$.\ However, if $r \geq n/2$ then $(2r+n)/(2n)$ does not lie between $0$ and $1$.\ In this case we replace the exponent by $(2r-n)/2n$.\ If we let $r_1$, $r_2$ and $r_3$ denote these normalized exponents, then it's known (e.g., Lemma 2.3 of \cite{Marks1})
that the minimal weight $k_0$ for $\rho$ satisfies $k_0 = 4(r_1+r_2+r_3)-2$.\ Note that in all cases, if $r$ is chosen as above then $r_2 = (n-r)/(2n)$, $r_3 = (2n-r)/(2n)$, and the following table summarizes the possibilities for $r_1$ and $k_0$:
\begin{center}
\begin{tabular}{r|c|c}
&$r_1$&$k_0$ \\
\hline
Cases A,C and $r \geq n/2$ &$\frac{2r-n}{2n}$&$2$\\
Cases A,C and $r < n/2$ &$\frac{2r+n}{2n}$&$6$\\
Cases B,D &$\frac{r}{n}$&$4$
\end{tabular}
\end{center}
The value of $r_1$ can be written uniformly as follows: set
\[
  e = \begin{cases}
-1 & \textrm{ in Cases A,C and } r \geq n/2,\\
0 & \textrm{ in Cases B,D,}\\
1 & \textrm{ in Cases A,C and } r < n/2.
\end{cases}
\]
then $r_1 = (2r+en)/(2n)$ and $k_0 = 2e+4$.\ The exponent differences satisfy:
\begin{align*}
r_1-r_2 &= \frac{3r+(e-1)n}{2n},&r_1-r_3 &= \frac{3r+(e-2)n}{2n}, & r_2-r_3 &= -\frac12.
\end{align*}
Thus, the only way that one of these differences can be an integer is if $n{\mid}3$.\ By Lemma \ref{lemred} these are precisely the cases where $\rho$ is reducible. Assume that $n \nmid 3$ and write 
\begin{align*}
  a&= r_1-\frac{k_0}{12}, & b &= r_2-\frac{k_0}{12},& c &= r_3-\frac{k_0}{12}.
\end{align*}
Then a basis (over the ring of classical scalar modular forms of level $1$) for the module of vector-valued modular forms for $\rho'$ is given by $(F, DF,D^2F)$, where $D = q\frac{d}{dq} - \frac{k}{12}E_2$ denotes the modular derivative of weight $k$, and where $F$ is the vector-valued modular form 
\begin{eqnarray}\label{Fdef}
F = \eta^{2k_0}\left(\begin{matrix}\ K^{a}\ _{3}F_2\left(a, a+\frac 13, a+\frac 23; a-b+1, a-c+1; K\right)\\K^{b}\ _{3}F_2\left(b, b+\frac 13, b+\frac 23; b-a+1, b-c+1; K\right)\\K^{c}\ _{3}F_2\left(c, c+\frac 13, c+\frac 23; c-a+1, c-b+1; K\right)
\end{matrix}\right).
\end{eqnarray}
Here $K = 1728j^{-1}$ and $j$ denotes the usual $j$-invariant of elliptic curves. Our first result on $F$ is the following.
\begin{lem}\label{lemcongcase} The following are equivalent.
\begin{enumerate}
\item[(i)] \emph{At least one} of the components of $F$ is a modular form on a congruence subgroup,
\item[(ii)] \emph{all} of the components of $F$ are modular forms on a congruence subgroup,
\item[(iii)] $\ker \rho'$ is a congruence subgroup,
\item[(iv)]$n{\mid}24$.
\end{enumerate}
In particular, all components of $F$ have \emph{bounded denominators} if any of these conditions holds.
\end{lem}
\begin{pf} Since $\rho$ and $\rho'$ are equivalent representations then $\rho=B\rho'B^{-1}$ for some invertible matrix $B$, and 
in particular $\ker\rho'=\ker\rho$.\ Then the equivalence of (iii) and (iv) follows from Theorem \ref{thmKlvl1}.\ On the other hand
we have
\begin{eqnarray*}
F|_{k_0}\gamma (\tau) = \rho'(\gamma)F(\tau)\ \ (\gamma \in \Gamma),
\end{eqnarray*}
so it is clear that the coordinates of $F$ are scalar modular forms on the finite index subgroup $\ker\rho'\subseteq \Gamma$.\ Therefore
(ii) and (iii) are equivalent.

Now suppose that some component $f$ of $F$ is a modular form on a congruence subgroup $\bar{\Gamma}(m)$, say, and assume to begin with
that $\rho$ (and therefore also $\rho'$) is \emph{irreducible}.\ Then the space of functions spanned by $f|_{k_0}\gamma\ (\gamma\in\Gamma)$ coincides with the span of
the coordinates of $F$.\ Since each $f|_{k_0}\gamma$ is modular on $\bar{\Gamma}(m)$ then the same is true for the coordinates of $F$.\
So (i)$\Rightarrow$(ii) in this case.\ On the other hand, if $\rho$ is \emph{not} irreducible then $n{\mid}3$ by Lemma \ref{lemred} and (iv) holds.\
The equivalence of (i)-(iv) follows from what we have established, and the Lemma is thus proved.
\end{pf}

In the remainder of this section we prove that when $n$ does \emph{not} divide $24$, all three coordinates of $F$ have unbounded denominators. As the details are largely the same in the three cases, we explain the argument for the first coordinate in full, and we suppress the details for the other two coordinates. 

\begin{rmk}
Chris Marks verified this result earlier in \cite{Marks} using a different but related method, in all but finitely many cases. The exceptional cases 
(there are probably many) were not made explicit in \cite{Marks}.
\end{rmk}

The hypergeometric series that occurs in the first coordinate of $F$ is
\[
\ _{3}F_2\left(a, a+1/3, a+2/3; a-b+1, a-c+1; K\right)=1+\sum_{m\geq 1}C_mj^{-m},
\]
where 
\begin{align*}
C_m &= \frac{12^{3m}}{m!}\cdot\frac{(a)_m(a+1/3)_m(a+2/3)_m}{(a-b+1)_m(a-c+1)_m} &=\  \frac{(e-1)n+3r}{m!}\cdot\left(\frac{2^8}{n}\right)^m\cdot \prod_{i=2m+e}^{3m+e-2}(ni+3r).
\end{align*}
Thus, the first coordinate of $F$ is the following modular form:
\[
f_1 \df \eta^{2k_0}j^{-\frac{(e-1)n+3r}{3n}}\left(1+((e-1)n+3r)\sum_{m\geq 1} \left\{\frac{2^{8m}\prod_{i=2m+e}^{3m+e-2}(in+3r)}{m!n^m}\right\}j^{-m}\right).
\]
Similarly, the second and third coordinates of the vector-valued modular form $F$ are equal to the functions
\begin{align*}
f_2 &\df \eta^{2k_0}j^{\frac {e-1}{6}+\frac r{2n}}\left(1-((e-1)n+3r)\sum_{m\geq 1} \left\{ \frac{2^{6m}\prod_{i=m+1}^{3m-1}((2i+1-e)n-3r)}{(2m)!n^{2m}} \right\}j^{-m}\right),\\
f_3&\df \eta^{2k_0}j^{\frac {e-4}6 + \frac{r}{2n}}\left(1 - ((e-4)n+3r)\sum_{m \geq 1}\left\{\frac{2^{6m}\prod_{i = m}^{3m-1} ((2i+4-e)n-3r)}{(2m+1)!n^{2m}}\right\}j^{-m}\right).
\end{align*}
\begin{thm}
\label{thm:firstcoord} 
Let $\rho = \Ind_{\bar\Gamma_0(2)}^\Gamma\chi$ be as in Theorem \ref{thmrho1}, write $\chi(\bar U) = e^{2\pi i r/n}$ where $0< r < n$ and $\gcd(r,n) = 1$.\ Assume that $n\nmid 3$.\ Let $k_0 = 2$, $4$ or $6$ denote the minimal weight such that there exist nonzero holomorphic vector-valued modular forms for $\rho$, write $k_0 = 2e+4$, and let $f_1$, $f_2$ and $f_3$ be defined as above.\ Write
\begin{align*}
f_1 &=\eta^{2k_0}\cdot\sum_{m \geq 0} \frac{a_m}{n^mm!}q^{m+\frac{r}{n}+\frac{e-1}{3}}, \\
f_2&=\eta^{2k_0}\cdot\sum_{m \geq 0} \frac{b_m}{n^{2m}(2m)!}q^{m+\frac{r}{n}+\frac{e}{3}-1}, \\
f_3 &=\eta^{2k_0}\cdot\sum_{m \geq 0} \frac{c_m}{n^{2m}(2m+1)!}q^{m+\frac{r}{n}+\frac{e}{3}-1}.
\end{align*}
Let $p$ be a prime dividing $n$. If $p\neq 2,3$ then all of $a_m$, $b_m$ and $c_m$ are congruent to $(24r)^m$ mod $p$. Otherwise one has:
\begin{align*}
a_m &\equiv\begin{cases} 
(6r)^m& \pmod{3^{m+1}} \quad  p = 3,~ 3^2 \mid n,\\
8^m& \pmod{2^{3m+1}}\quad  p = 2,~ 2^4 \mid n,
\end{cases} \\
 b_m,c_m &\equiv \begin{cases} 
(3r)^{2m}&\pmod{3^{2m+1}} \quad p = 3,~ 3^2 \mid n,\\
2^{6m} &\pmod{2^{6m+1}}\quad p = 2,~ 2^4 \mid n.
\end{cases}
\end{align*}
Thus if $n\nmid 24$ then $f_1$, $f_2$ and $f_3$ all have unbounded denominators.
\end{thm}
\begin{proof}
We give the details of the proof for $f_1$.\ The other cases are analogous.\ Write $j^{-1}=q(1+qg(q))$ with $g(q)\in \pseries{\ZZ}{q}$ of the form $-744 + 356652q + O(q^2)$.\ Then
\[
j^{-\frac{3r+(e-1)n}{3n}}=q^{\frac{r}{n}+\frac{e}{3}-1}\left(1+\sum_{x\geq 1}\frac{\prod_{i = 0}^{x-1}(3r+(2i+e-1)n)}{(3n)^xx!}(qg(q))^x\right)
\]
and thus
\begin{align*}
\frac{f_1}{\eta^{2k_0}} &= q^{\frac{r}{n}+\frac{e-1}{3}}\left(1+\sum_{x\geq 1}\frac{\prod_{i = 0}^{x-1}(3r+(2i+e-1)n)}{(3n)^xx!}(qg(q))^x\right)\times \\
&\left(1+((e-1)n+3r)\sum_{m\geq 1} \left\{\frac{2^{8m}\prod_{i=2m+e}^{3m+e-2}(in+3r)}{m!n^m}\right\}(1+qg(q))^mq^{m}\right).
\end{align*}
This gives the following $q$-expansion for $f_1\eta^{-2k_0}q^{-\frac rn-\frac {e-1}3}$:
\begin{align*}
&1 + \sum_{x\geq 1}\frac{\prod_{i = 0}^{x-1}(3r+(2i+e-1)n)}{(3n)^xx!}(qg(q))^x+((e-1)n+3r)\times \\
&\left(\sum_{m\geq 1}\sum_{y = 0}^m \binom{m}{y}\left\{\frac{2^{8m}\prod_{i=2m+e}^{3m+e-2}(in+3r)}{m!n^m}\right\} q^{m+y}g(q)^y +\right.\\
&\left.\sum_{x\geq 1}\sum_{m\geq 1}\sum_{y = 0}^m\binom{m}{y}\left\{\frac{2^{8m}\prod_{i = 0}^{x-1}(3r+(2i+e-1)n) \prod_{i=2m+e}^{3m+e-2}(in+3r)}{3^xn^{m+x}m!x!}\right\} q^{m+x+y}g(q)^{x+y}\right).
\end{align*}
Suppose that $p{\mid}n$ and $p\neq 2,3$ or $31$ (the prime divisors of $744$). Then the coefficient of $q^m$ in this expression is a rational number of the form
\begin{align*}
&\frac{(-248)^m\prod_{i = 0}^{m-1}(3r+(2i+e-1)n)}{n^mm!}+((e-1)n+3r)\times \left(\frac{2^{8m}\prod_{i=2m+e}^{3m+e-2}(ni+3r)}{n^mm!} +\right.\\
&\left.\sum_{s =1}^{m-1}\binom{m}{s}\frac{(-248)^{s}2^{8(m-s)}\prod_{i = 0}^{s-1}(3r+(2i+e-1)n)\prod_{i=2(m-s)+e}^{3(m-s)+e-2}(ni+3r)}{n^{m}m!}\right)
\end{align*}
plus terms with lower powers of $p$ in the denominator. This shows that this coefficient is of the form $\frac{a_t}{n^tt!}$ where $a_t$ is an integer congruent mod $p$ to
\[
(-744r)^t+(3r)^t\left(2^{8t} +\sum_{s =1}^{t-1}\binom{t}{s}(-248)^{s}2^{8(t-s)}\right)\equiv (24r)^t\pmod{p}.
\]
This proves the theorem for $p\neq 2,3,31$.

If $p = 31$ then the coefficient of $q^m$ in the expression for $f_1\eta^{-2k_0}q^{-\frac rn-\frac {e-1}3}$ is a rational number of the form $\alpha/n^mm!$, where
\[
\alpha = ((e-3)n+3r)\left(2^{8m}\prod_{i=2m+e}^{3m+e-2}(ni+3r) \right)\equiv (768r)^{m} \equiv (24r)^m \pmod{31},
\]
plus a rational number with lower powers of $31$ in the denominator. This proves the theorem in this case.

Next consider $p = 2$ and suppose that $2^4{\mid}n$. Write $n =2^4\eta$. Then in this case the $q$-expansion that we're interested in is
\begin{align*}
&1 + \sum_{x\geq 1}\frac{31^x\prod_{i = 0}^{x-1}(3r+(2i+e-1)2^4\nu)}{(2\nu)^xx!}(qg'(q))^x+((e-1)2^4\nu+3r)\times \\
&\left(\sum_{m\geq 1}\sum_{y = 0}^m \binom{m}{y}\left\{\frac{(93)^y\prod_{i=2m+e}^{3m+e-2}(2^4i\nu+3r)}{m!\nu^m}2^{4m+3y}\right\}q^{m+y}g'(q)^y +\sum_{x\geq 1}\sum_{m\geq 1}\sum_{y = 0}^m\right.\\
&\left.\binom{m}{y}\frac{(93)^{x+y}2^{3y+4m-x}\prod_{i = 0}^{x-1}(3r+(2i+e-1)2^4\nu)\prod_{i=2m+e}^{3m+e-2}(2^4i\nu+3r)}{3^x\nu^{m+x}m!x!} q^{m}(qg'(q))^{x+y}\right).
\end{align*}
where $g' = g/744$. In this case we see that the coefficient of $q^m$ is of the form
\[
\frac{(-31)^m\prod_{i = 0}^{m-1}(3r+(2i+e-1)2^4\nu)}{(2\nu)^mm!}
\]
plus a rational number with fewer powers of $2$ in its denominator. Thus, the coefficient is of the form $\alpha/(2\nu)^mm!$ where $\alpha$ is an odd integer, and so if we write it in the form $a_m/n^mm!$ we see that $a_m = 2^{3m}\alpha \equiv 2^{3m} \pmod{2^{3m+1}}$.

Finally consider $p = 3$ and suppose that $3^2{\mid}n$. Write $n = 3^2\nu$, so that in this case we're interested in
\begin{align*}
&1 + \sum_{x\geq 1}\frac{248^x\prod_{i = 0}^{x-1}(r+(2i+e-1)3\nu)}{(3\nu)^xx!}(qg'(q))^x+((e-1)3\nu+r)\times \\
&\left(\sum_{m\geq 1}\sum_{y = 0}^m \binom{m}{y}\frac{2^{8m}248^{y}3^y\prod_{i=2m+e}^{3m+e-2}(3i\nu+r)}{m!(3\nu)^m} q^{m+y}g'(q)^y +\sum_{x\geq 1}\sum_{m\geq 1}\sum_{y = 0}^m\right.\\
&\left.\binom{m}{y}\frac{2^{8m}\prod_{i = 0}^{x-1}(r+(2i+e-1)3\nu) \prod_{i=2m+e}^{3m+e-2}(3i\nu+r)248^{x+y}}{3^{m+x-y}\nu^{m+x}m!x!} q^{m}(qg'(q))^{x+y}\right).
\end{align*}
In this case we see that the coefficient of $q^m$ is of the form
\begin{align*}
&\frac{(-248)^m\prod_{i = 0}^{m-1}(r+(2i+e-1)3\nu)}{(3\nu)^mm!}+((e-1)3\nu+r)\times \left(\frac{2^{8m}\prod_{i=2m+e}^{3m+e-2}(3i\nu+r)}{m!(3\nu)^m} \right. \\
&\left. +\sum_{x=1}^{m-1}\binom{m}{x}\frac{2^{8(m-x)}(-248)^{x}\prod_{i = 0}^{x-1}(r+(2i+e-1)3\nu) \prod_{i=2(m-x)+e}^{3(m-x)+e-2}(3i\nu+r)}{(3\nu)^{m}m!}\right).
\end{align*}
plus a rational number with fewer powers of $3$ in the denominator. We thus see that the coefficient is of the form $\alpha/(3\nu)^mm!$ where $\alpha$ is an integer satisfying
\[
  \alpha \equiv r^m\left(1+2^{8m} +\sum_{x=1}^{m-1}\binom{m}{x}2^{8(m-x)}\right) \equiv (2r)^m \pmod{3}
\]
Thus, if we write the coefficient as $a_m/n^mm!$ then $a_m \equiv (6r)^m \pmod{3^{m+1}}$. This concludes the proof of the congruences for $f_1$, and the final claim follows immediately from them.
\end{proof} 

\section{Unbounded denominators in the general case}
In this Section, we consider the question of unbounded denominators for vector-valued modular forms of \emph{arbitrary weight}
associated with the induced representation $\rho$ (cf.\ Proposition \ref{thmrho1}).\ Stated in terms of modular forms on
$\bar{\Gamma}_0(2)$, the main result is as follows.
 \begin{thm}\label{thmgenubd} Let $f\in M_k(\bar{\Gamma}_0(2), \chi)$ be a nonzero holomorphic modular form, where $\chi$ is as in (\ref{chidef}).\
 Suppose that $f$ has algebraic Fourier coefficients.\  If $p$ is a prime
 dividing $n$, the powers of $p$ that divide the denominators of $f$
 are \emph{unbounded} under any of the following circumstances: $p\geq 5$; $p=3$ and $p^2{\mid}n$; $p=2$ and $p^4{\mid}n$.\ In particular, the following are equivalent:
\begin{enumerate}
\item[(a)] $f$ has \emph{bounded denominators}; 
\item[(b)] $n{\mid}24$; 
\item[(c)] $f$ is a congruence modular form.\ 
\end{enumerate}
\end{thm}

By Theorem \ref{thm:firstcoord}, the conclusions of the Theorem hold if, in place of $f$, we take one of the components $f_1, f_2$ or $f_3$ of the vector-valued modular form $F$ that we studied in the previous Section.\ The idea of the proof of Theorem \ref{thmgenubd} is to transfer this special result for $F$ into a general result about vector-valued modular forms.\ The methods for doing this are essentially already in the literature (\cite{M2}, \cite{MM}, \cite{FM1}), though they are not stated in the form that we need here.\  We will therefore give some of the details.\ 

First we must revert to the original induced representation $\rho$ described in Proposition \ref{thmrho1} from the equivalent representation
$\rho'$  used in Section \ref{s:newvvmfs}. We always assume that $n$ does \emph{not} divide $3$, so that $\rho$ is irreducible (Lemma \ref{lemred}).\ Let 
\begin{eqnarray*}
\mathcal{H}(\rho) = \oplus_{k\geq k_0} \mathcal{H}(k, \rho)
\end{eqnarray*}
be the $\ZZ$-graded space of vector-valued modular forms associated with $\rho$.\ Similarly, we have $\mathcal{H}(\rho')$.\ If $B$ is an invertible matrix that intertwines $\rho$ and $\rho'$, i.e., $B\rho'B^{-1}=\rho$, then the map $F\mapsto BF$ induces an isomorphism of $\ZZ$-graded spaces $\mathcal{H}(\rho')\stackrel{\cong}{\longrightarrow} \mathcal{H}(\rho)$.\
We have already pointed out that
the minimal weight of a nonzero form in $\mathcal{H}(\rho')$ is $k_0=2, 4$ or $6$, so the same is true for $\mathcal{H}(\rho)$.\
We take the matrix $B$ to be
\begin{eqnarray*}
B\df \left(\begin{matrix} 1 & 0 & 0 \\0 & 1 & 1 \\0 & \epsilon\sigma & -\epsilon\sigma\end{matrix}\right),
\end{eqnarray*}
so that $B\rho'(\bar{T})B^{-1} = \rho(\bar{T})$.\ Because 
$F\df\left(\begin{smallmatrix}f_1 \\ f_2 \\ f_3\end{smallmatrix}\right)$ satisfies $F\in\mathcal{H}(k_0, \rho')$ then
$BF\in\mathcal{H}(k_0, \rho)$, and 
$BF =\left(\begin{smallmatrix} f_1 \\ f_2+f_3 \\ \epsilon\sigma(f_2-f_3)\end{smallmatrix}\right)$.\ From Theorem \ref{thm:firstcoord} we immediately deduce
\begin{lem}
\label{lempF0} 
Let $F_0\df BF$ be as above.\ Then $F_0$ has algebraic Fourier coefficients.\ Moreover, if $p$ is a prime dividing $n$, the powers of $p$ that divide the denominators of the Fourier coefficients of each of the components of $F_0$ are \emph{unbounded} under any of the following circumstances: $p\geq 5$; $p=3$ and $p^2 \mid n$; $p=2$ and $p^4\mid n$. $\hfill \Box$
\end{lem}

For a field $E\supseteq \QQ$ we write
\begin{eqnarray}
\label{spvvforms}
\mathcal{H}_{E}(\rho) := \oplus_{k\geq k_0} \mathcal{H}_{E}(k, \rho),
\end{eqnarray}
for the $\ZZ$-graded space of vector-valued modular forms  whose Fourier coefficients lie in $E$.\
Let $\frak{M}_{E}$ be the space of classical modular forms on $\Gamma$ with Fourier coefficients in $E$, and let
\begin{eqnarray*}
D: \mathcal{H}(k, \rho) \rightarrow \mathcal{H}(k+2, \rho)
\end{eqnarray*}
be the usual differential operator $D=D_k: f\mapsto q\frac{df}{dq} -\frac{kE_2}{12}f.$
The ring of differential operators $\frak{M}_{E}(D)$ consists of  (noncommutative) polynomials in $D$ with coefficients in $\frak{M}_E$
and satisfying $Df-fD=D(f)$.\ This ring acts naturally on $\mathcal{H}_{E}(\rho)$, with $\frak{M}_E$ acting componentwise.

 We  now prove
 \begin{thm}\label{thmvvmfdim3} Suppose that  $F\in \mathcal{H}(k, \rho)$ has \emph{algebraic} coefficients.\
 If $p$ is a prime
 dividing $n$, the powers of $p$ that divide the denominators of the Fourier coefficients of each of the components of $F$
 are \emph{unbounded} under any of the following circumstances: $p\geq 5$; $p=3$ and $p^2\mid n$; $p=2$ and $p^4\mid n$.
 \end{thm}
 \begin{pf} (Cf.\ Section 3 of \cite{FM1}.)\ We use
 the following assertion: ($\star$)\ if $I\subseteq \mathcal{H}_E(\rho)$ is a \emph{nonzero} $\frak{M}_{E}(D)$-submodule, and if $\Delta=\eta^{24}$ is the  usual discriminant, then some power $\Delta^r$ of $\Delta$ \emph{annihilates} $\mathcal{H}_E(\rho)/I$.\ Indeed, the proof given in \cite{FM1} for the case $E=\QQ$ applies in general.
 
 We  take $E=\bar\QQ$ with $I$  the set of forms $G\in \mathcal{H}_{\bar \QQ}(\rho)$  that do \emph{not} satisfy the conclusions of the Theorem.\ So there is an index $j$ and a large enough integer $N$ (depending on $G$ and $j$) such that the Fourier coefficients of the $j^{th}$ component $G_j$ of $G$ is such that $p^NG_j$ is \emph{$p$-integral}.\ Here,  $p$ is a prime divisor of $n$, as in the statement of Lemma \ref{lempF0}.
 
 Now $I$ is an $\frak{M}_{\bar\QQ}(D)$-submodule of $\mathcal{H}_{\bar\QQ}(\rho)$, so by ($\star$) we can conclude that either $I=0$ or there is an integer $r$ such that $\Delta^rF_0\in I$, where $F_0$ is as in Lemma \ref{lempF0}.\ In the latter case, there is a power $p^N$ of $p$ and an index $j$ such that $p^N\Delta^r(F_0)_j$ is $p$-integral,
 in which case $p^N(F_0)_j$ is itself $p$-integral.\ Because this contradicts Lemma \ref{lempF0}, the conclusion is that $I=0$.\ The Theorem follows immediately.
 \end{pf}
 
 Finally, suppose that $f\in M_k(\bar{\Gamma}_0(2), \chi)$ is a holomorphic modular form with algebraic Fourier coefficients.\ By Proposition \ref{thmrho1}, $f$ is the first component of a  vector-valued modular form $F'\in \mathcal{H}_{\bar\QQ}(k, \rho)$, and we can then apply Theorem \ref{thmvvmfdim3} to $F'$ to complete the proof of Theorem \ref{thmgenubd}.  

\section{Congruences}
Let $n$ be a positive integer and let $0 < r < n$ be another integer.\ Define the character $\chi_{n,r}\colon \bar \Gamma_0(2) \to \CC^\times$ by setting $\chi_{n,r}(\bar U) = e^{2\pi i r/n}$ and $\chi_{n,r}(\bar V) = -1$.\ Assume that $n$ is odd and $r > n/2$, so that we are in Case A with a minimal weight of $2$ for $\rho = \Ind\chi$.\ In this case $H_0 = \ker \chi = \langle \bar U^n\rangle \Gamma_0(2)'$.\ Set $X = X_{H_0}$, which is a hyperelliptic curve of genus $g_X = (n-1)/2$
(cf.\ Theorem \ref{thmhypereqn}).\ One can use Riemann-Roch to show that $S_2(\bar\Gamma_0(2),\chi_{n,r})$ is one dimensional for each such pair $(n,r)$, and we have
\[
S_2(H_0) = \bigoplus_{r = (n+1)/2}^{n-1} S_2(\bar \Gamma_0(2),\chi_{n,r}).
\]
By looking at the first coordinate of the corresponding vector-valued modular forms, one obtains a basis for $S_2(H_0)$ as follows.\footnote{Let $\rho'$ be a representation equivalent to $\rho$ such that $\rho(\bar T)$ is diagonal.\ It was observed above that the first coordinate of a vector-valued modular form for $\rho$ is a modular form on $H_0$.\ This will still be true for the eqiuvalent representation $\rho'$, and it is this particular representation that is used to prove Theorem \ref{t:basis}.}
\begin{thm}
\label{t:basis}
Let $n$ be an odd integer that does not divide $3$ and set $H_0 = \langle \bar U^n\rangle\bar\Gamma_0(2)'$.\ Then $\dim_\CC S_2(H_0) = (n-1)/2$ and a basis for this space is given by the modular forms
\[
  f_{n,r} \df\eta^4K^{\frac rn - \frac 23}\ _3F_2\left(\frac{r}{n}-\frac 23,\frac rn-\frac 13, \frac rn;\frac{3r}{2n},\frac{3r}{2n} - \frac 12;K\right)
\]
where $r$ runs from $(n+1)/2$ up to $n-1$.
\end{thm}

After Theorem \ref{thmhypereqn} we may also interpret these modular forms as a basis of holomorphic differentials on the smooth projective hyperelliptic curve defined by the affine equation $y^2=x^n+64$.

\begin{ex}
The geometric considerations of this paper were inspired by the following computation: let $n = 5$. Then one checks that 
\[
  \dim_\CC S_k(H_0) = \begin{cases}
  2 & k=2,\\
  6 & k=4. 
\end{cases}
\]
Since $X$ is hyperelliptic, the canonical embedding is not a closed embedding in this case.\ To find a projective model for $H_0\backslash \uhp^*$, one can instead use modular forms of weight $4$.

To describe $S_2(H_0)$ we use the characters $\chi_{(5,3)}$ and $\chi_{(5,4)}$.\ Let $f_1$ and $f_2$ denote the corresponding cusp forms.\ One finds that
\begin{align*}
f_1 &= \eta^4K^{-\frac{1}{15}}\ _3F_2\left(-\frac{1}{15},\frac{4}{15},\frac 35;\frac{9}{10},\frac 25;K\right)\\
    &= q_{10} - \frac{28}{5}q_{10}^{11} + \frac{222}{5^2}q_{10}^{21} + \frac{168}{5^3}q_{10}^{31}-\frac{5071}{5^4}q_{10}^{41}-\frac{123732}{5^6}q_{10}^{51} - \frac{58634}{5^7}q_{10}^{61} + \cdots, \\
f_2 &= \eta^4K^{\frac{2}{15}}\ _3F_2\left(\frac{2}{15},\frac{7}{15},\frac 45;\frac 65,\frac{7}{10};K\right)\\
&=q_{10}^3 - \frac{4}{5}q_{10}^{13} - \frac{102}{5^2}q_{10}^{23} + \frac{296}{5^3}q_{10}^{33}+\frac{1839}{5^4}q_{10}^{43}+\frac{15324}{5^6}q_{10}^{53}  + \frac{463134}{5^7}q_{10}^{63} +\cdots,
\end{align*}
where $q_{10} = e^{2\pi i z/10}$. One can similarly show that
\begin{align*}
  f_3 &= \eta^4K^{-\frac 16}~_2F_1(-1/6,1/6;1/2;K),\\
  f_4 &= \eta^4K^{1/3}~_2F_1(1/3,2/3;3/2;K),
\end{align*}
describe a basis for $M_2(\bar\Gamma(2))$ (see Example 21 of \cite{FM2}), and the forms $f_1$, $f_2$, $f_3$ and $f_4$ define a basis for $M_2(H_0)$.\ The forms
\begin{align*}
G_1 &= f_1^2, & G_2 &= f_1f_2, & G_3 &= f_2^2,\\
G_4 &= f_1f_3, & G_5 &= f_2f_3, & G_6 &= f_2f_4. 
\end{align*}
then yield a basis for $S_4(H_0)$, and it defines a closed embedding $X \to \PP^5_\CC$ via
\[z\mapsto (G_1(z):G_2(z):G_3(z):G_4(z):G_5(z):G_6(z)).\]
The image is the smooth curve whose zero locus is defined by the homogeneous equations
\[
\left(\begin{array}{r}
X_{1}^{2} -  X_{4} X_{5} + 64 X_{3} X_{6} \\
X_{1} X_{2} -  X_{5}^{2} + 64 X_{6}^{2} \\
- X_{2}^{2} + X_{1} X_{3} \\
- X_{2} X_{4} + X_{1} X_{5} \\
- X_{2} X_{3} + X_{1} X_{6} \\
- X_{3} X_{4} + X_{2} X_{5} \\
- X_{3}^{2} + X_{2} X_{6} \\
X_{3} X_{5} -  X_{4} X_{6}
\end{array}\right).
\]
In the affine chart defined by $X_6 \neq 0$ the curve is given by
\[
\{(a^3:a^2:a:ab:b:1) \in \PP^5 \mid a^5-b^2+64 = 0\},
\]
where $a = X_3/X_6$ and $b = X_5/X_6$.\ This is as expected by Theorem \ref{thmhypereqn}.
\end{ex}

Before we state and prove our congruence result we must prepare with some preliminary definitions and results.\ Let $q = p^r$ be a prime power and let $\chi$ denote a character of the cyclic group $\FF_{q}^\times$.\ Extend $\chi$ to a function on $\FF_{q}$ by setting $\chi(0) = 0$ if $\chi \neq 1$ and $\chi(0) = 1$ otherwise.\ Recall that the \emph{Gauss sum} $G(\chi)$ of such a character $\chi$ is the complex number
\[
  G(\chi) \df \sum_{u \in \FF_q}\chi(u)\be(\Tr_{\FF_q/\FF_p}(u)/p). 
\]
If $\chi'$ is a second character of $\FF_q^\times$, then the \emph{Jacobi sum} $J(\chi,\chi')$ associated to $\chi$ and $\chi'$ is the complex number
\[
  J(\chi,\chi') \df \sum_{u_1+u_2 = 1}\chi(u_1)\chi'(u_2).
\]
\begin{lem}
\label{l:jsum}
Let $\chi$ and $\chi'$ be characters of $\FF_q^\times$. If $\chi\chi'$ is nontrivial then
\[
  J(\chi,\chi')= \frac{G(\chi)G(\chi')}{G(\chi\chi')}.
\]
\end{lem}
\begin{proof}
Theorem 2.1.3 of \cite{BEW}.
\end{proof}
\begin{lem}
\label{l:gsum}
Let $n$ be an odd positive integer and let $p$ be an odd prime congruent to $-1$ mod $n$. Let $q = p^{2t}$, and let $\chi$ be a character of $\FF_q^\times$ of order $n$. Then the Gauss sum $G(\chi)$ satisfies $G(\chi) = (-1)^{t+1}p^t$.
\end{lem}
\begin{proof}
See Theorem 11.6.3 of \cite{BEW}.
\end{proof}
\begin{prop}
\label{p:lfunction}
  Let $n$ be an odd integer and let $p$ be a prime satisfying $p \equiv -1\pmod{n}$. Then the numerator of the zeta function of a smooth projective model of the affine curve $y^2 = x^n+64$ is equal to $L_p(T) = (1+pT^2)^{(n-1)/2}$.
\end{prop}
\begin{proof}
Use results of Weil \cite{W1} and the preceding results on Gauss and Jacobi sums.
\end{proof}
An $n$-term congruence result for weight two modular forms on $X$ follows immediately from Proposition \ref{p:lfunction} and Theorem 6.1 of \cite{K1}.\ The factorization of the $L$-function of $X$ into quadratic factors suggests that one might be able to reduce this to a three-term congruence relation of Hecke type.\ Our next theorem confirms that such congruence relations indeed hold.
\begin{thm}
\label{t:ASD}
Let $n$ be an odd integer that does not divide $3$, let $r$ denote an integer between $(n+1)/2$ and $n-1$, and let $f_{n,r}$ be a modular form of weight $2$ on $X$ as in Theorem \ref{t:basis}.\ Let $N = 2n$ and let $f_{n,r}(q_N) = \sum_{m \geq 1} a_mq^m_N$ denote the $q_N$-expansion of $f_{n,r}$, where $q_N= e^{2\pi i z/N}$.\ Then for all primes $p\equiv -1\pmod{n}$, and for all indices $m \geq 1$, one has the congruence
\[
  a_{p^2m}+pa_m \equiv 0 \pmod{p^{2 + v_p(m)}}.
\]
\end{thm}
\begin{proof}
If $\gcd(n,r) > 1$ then we may replace $n$ and $r$ by $n/\gcd(n,r)$ and $r/\gcd(n,r)$ and work on the curve $y^2 = x^{n/\gcd(n,r)}+64$.\ We may thus assume without loss of generality that $\gcd(n,r) = 1$.

Let $G = \bar{\Gamma}_0(2)/H_0 \cong (\ZZ/n\ZZ)\times (\ZZ/2\ZZ)$ and let $\zeta$ denote a primtive $n$th root of unity.\ Then $G$ acts on $X$ via $(r,e)\cdot (x,y) = (\zeta^rx,ey)$. Recall that there is a decomposition
\[
S_2(H_0) = \bigoplus_{r = (n+1)/2}^{n-1} S_2(\bar \Gamma_0(2),\chi_{n,r})
\]
where $S_2(\bar \Gamma_0(2),\chi_{n,r})$ is the $\chi_{n,r}$-isotypic piece of $S_2(H_0)$ under the induced action of the finite abelian group $G$.\ Furthermore, each of these pieces is one-dimensional.\ The idea now is to use this group action and Theorem 6.1 of \cite{K1}.

Let $p \equiv -1\pmod{n}$ be a prime, so that if $q = p^2$, then $\FF_q$ contains a primitive $n$th root of unity.\ Thus, if we now write $X_q$ for the reduction of $X$ mod $p$, basechanged to $\FF_q$, then $G$ acts on $X_q$ over $\FF_q$.\ Let $F$ denote the $q$th power Frobenius.\ We wish to compute the twisted $L$-series
\begin{align*}
  L(X_q,\chi_{n,r},T) &\df \exp\left(\frac{1}{2n}\sum_{m\geq 1} \frac{T^m}{m}\sum_{g \in G} \Tr(\chi_{n,r}(g^{-1}))\abs{X(\bar \FF_q)^{F^mg}}\right)\\
&= \exp\left(\frac{1}{2n}\sum_{m\geq 1} \frac{T^m}{m}\sum_{d \mid n}\sum_{\substack{a = 1\\ \gcd(a,\frac{n}{d}) = 1}}^{n/d}\frac{\left\{\abs{X(\bar \FF_q)^{F^m(ad,0)}}-\abs{X(\bar \FF_q)^{F^m(ad,1)}}\right\}}{\zeta^{adr}}\right).
\end{align*}
Note that $q \equiv 1 \pmod{n}$ and $q$ odd imply
\[
X(\bar \FF_q)^{F^m(ad,e)} = \{(x,y) \in X(\bar \FF_q) \mid (\zeta^{ad}F^m(x),(-1)^eF^m(y)) = (x,y)\}\cup\{\infty\}.
\]
These conditions force $x \in \FF_{q^{mn/d}}$ and $y \in \FF_{q^{2m/\gcd(2,e)}}$. Let us now set
\[
  N(m,ad,e) = \abs{\{(x,y) \in \bar \FF_q^2 \mid y^2 = x^n+64 \text{ and } (\zeta^{ad}F^m(x),(-1)^eF^m(y)) = (x,y)\}}
\]
so that
\begin{align*}
 & L(X_q,\chi_{n,r},T) = \exp\left(\frac{1}{2n}\sum_{m\geq 1} \frac{T^m}{m}\sum_{d \mid n}\sum_{\substack{a = 1\\ \gcd(a,n/d) = 1}}^{n/d}\zeta^{-adr}\left\{N(m,ad,0)-N(m,ad,1)\right\}\right)\\
&=\exp\left(\frac{1}{n}\sum_{m\geq 1} \frac{T^m}{m}\sum_{d \mid n}\sum_{\substack{a = 1\\ \gcd(a,n/d) = 1}}^{n/d}\zeta^{-adr}\sum_{\substack{u_1-u_2=-64\\ u_1,u_2 \in \FF_{q^m}}}N(x^n = u_1; F^m(x)=\zeta^{-ad}x)\kappa(u_2)\right)
\end{align*}
where $\kappa$ denotes the quadratic character of $\FF_{q^m}^\times$. If $x^n = u_1$ and $F^m(x)=\zeta^{-ad}x$, then the other solutions of $x^n = u_1$ are the $\zeta^jx$, and $F^m(\zeta^jx) = \zeta^{j-ad}x$.\ Note that $\zeta^{-ad}$ is a primitive $(n/d)$th root of unity.\ It follows that the condition $F^m(x) = \zeta^{-ad}x$ forces $\FF_{q^{mn/d)}}= \FF_{q^m}(x)$.\ These observations show that
\[
N(x^n = u_1; F^m(x)=\zeta^{-ad}x) = \frac{1}{\phi(n/d)}N(x^n = u_1; \FF_{q^{mn/d}} = \FF_{q^m}(x)).
\]
In particular, this quantity is independent of $a$.\ Note that if $m$ is an integer, and if $\zeta_m$ is a primitive $m$th root of unity, then
\[
\sum_{\substack{a = 1\\ \gcd(a,m) = 1}}^{m}\zeta^{-ar}_m = \sum_{a = 1}^{m}\sum_{b \mid \gcd(a,m)}\mu(b)\zeta^{-ar}_m = \sum_{b \mid m}\mu(b)\sum_{a = 1}^{m/b}\zeta^{-abr}_m = m\sum_{\substack{b\mid m\\ m\mid br}}\frac{\mu(b)}{b}
\]
where $\mu$ is the Moebius function.\ We thus see that $L(X_q,\chi_{n,r},T)$ is equal to
\[\exp\left(\frac{1}{n}\sum_{m\geq 1} \frac{T^m}{m}\sum_{d \mid n}\frac{\mu(d)}{\phi(d)}\sum_{\substack{u_1-u_2=-64\\ u_1,u_2 \in \FF_{q^m}}}N(x^n = u_1; \FF_{q^{md}} = \FF_{q^m}(x))\kappa(u_2)\right);\]
note that the condition $d{\mid}b$ replaced $d{\mid}br$, which is permissible since $\gcd(n,r) = 1$. 

Define an arithmetic function $\alpha$ as $\alpha(x) = \prod_{p\mid x} (-1)^{v_p(x)}$.\ An inclusion-exclusion argument allows one to eliminate the condition $\FF_{q^{md}} = \FF_{q^m}(x)$.\ We deduce that
\begin{align*}
&L(X_q,\chi_{n,r},T)\\
&=\exp\left(\frac{1}{n}\sum_{m\geq 1} \frac{T^m}{m}\sum_{d \mid n}\frac{\mu(d)}{\phi(d)}\sum_{\substack{u_1-u_2=-64\\ u_1,u_2 \in \FF_{q^m}}}\sum_{c \mid d}\alpha(d/c)N(x^n = u_1; x\in \FF_{q^{mc}})\kappa(u_2)\right).
\end{align*}
We have $N(x^n = u_1; x \in \FF_{q^{mc}}) = \sum_{\chi^n = 1}\chi(u_1)$, where the sum runs over all characters of $\FF_{q^{mc}}^\times$ of order dividing $n$. Since $u_1$ is in $\FF_{q^m}$, we'd like to sum over characters of $\FF_{q^m}$ instead.\ The characters of $\FF_{q^{mc}}^\times$ of order $n$ restrict to characters of $\FF_{q^m}^\times$ of order $n/c$. 
Applications of Lemmas \ref{l:jsum} and \ref{l:gsum} yield
\[
\sum_{\chi^{n/c}=1}\sum_{\substack{u_1-u_2=-64\\ u_1,u_2 \in \FF_{q^m}}}\chi(u_1)\kappa(u_2) = \sum_{\chi^{n/c} = 1}(-1)^{m+1}p^m,
\]
so that we deduce
\[
L(X_q,\chi_{n,r},T)=\exp\left(-\sum_{m\geq 1} \frac{(-pT)^m}{m}\sum_{d \mid n}\frac{\mu(d)}{\phi(d)}\sum_{c \mid d}\alpha(d/c)\right).
\]
Note that
\[
\sum_{c \mid d}\alpha(c) = \prod_{p \mid d}\sum_{t = 0}^{v_p(d)}\alpha(p^t) = \prod_{p \mid d}\frac{(-1)^{v_p(d)}+1}{2} = \begin{cases}
1 & d = z^2,\\
0 & d \neq z^2.
\end{cases}
\]
Since the Moebius function $\mu$ vanishes on squares, save for $\mu(1) = 1$, we conclude that $L(X_q,\chi_{n,r},T) = 1+pT$.\ Theorem 6.1 of \cite{K1} now implies that $a_{mp^2}+pa_m \equiv 0 \pmod{p^{2+v_p(m)}}$ for all $m \geq 1$.
\end{proof}

\begin{rmk}
It would be of interest to find a proof of Theorem \ref{t:ASD} that uses our explicit formulae for the modular forms $f_{n,r}$ in place of the use of Theorem 6.1 of \cite{K1}. Presumably such a proof would then generalize to Cases B and C of this paper, where the minimal weights are $4$ and $6$.\ For this it might be useful to note that $\eta(6\tau)^4$ is the newform corresponding to the elliptic curve $y^2 = x^3+1$ of conductor $36$, whose $q$-expansion is well-understood (cf. Proposition 8.5.3 of \cite{C}).
\end{rmk}

\bibliographystyle{plain}

\end{document}